\documentclass[11pt]{amsart}

\title[Thurston's weak metric]{Tangent spaces of the Teichm\"{u}ller space of the torus with Thurston's weak metric}
\author[Miyachi, Ohshika and Papadopoulos]{Hideki Miyachi, Ken'ichi Ohshika, and Athanase Papadopoulos}
\address{Hideki Miyachi,
School of Mathematics and Physics,
College of Science and Engineering,
Kanazawa University,
Kakuma-machi, Kanazawa,
Ishikawa, 920-1192, Japan}
\email{miyachi@se.kanazawa-u.ac.jp} 
\address{Ken'ichi Ohshika,
Department of Mathematics,
Gakushuin University,
Mejiro, Toshima-ku, Tokyo, Japan}
  \email{ohshika@math.gakushuin.ac.jp}
\address{Athanase Papadopoulos, 
Institut de Recherche Math\'ematique Avanc\'ee (Universit\'e de Strasbourg et CNRS),
7 rue Ren\'e Descartes
67084 Strasbourg Cedex France}
 \email{papadop@math.unistra.fr}
\date{\today}

\thanks{This work is supported by JSPS KAKENHI Grant Numbers
16K05202, partially, 16H03933, 17H02843}

\makeatletter
\@addtoreset{equation}{section}

\makeatother

\usepackage{amsfonts, amsmath, amssymb, amsthm}
\usepackage{url}
\usepackage[dvipdfmx]{graphicx}
\usepackage[dvipdfmx]{color}
\usepackage{amscd}
\usepackage[right]{lineno}
\usepackage[capitalize]{cleveref}
\modulolinenumbers[5]

\usepackage{enumerate}

\newtheorem{theorem}{Theorem}[section]
\newtheorem{lemma}{Lemma}[section]
\newtheorem{proposition}{Proposition}[section]
\newtheorem{corollary}{Corollary}[section]

\newcommand{\teich}{\mathcal{T}}

\newcommand{\torus}{T^2}

\newcommand{\SSS}{\mathcal{S}}

\newcommand{\hyperbolic}{\mathbb{H}}
\newcommand{\complexes}{\mathbb{C}}
\newcommand{\reals}{\mathbb{R}}
\newcommand{\integers}{\mathbb{Z}}
\renewcommand{\l}{\operatorname{length}}
\renewcommand{\Im}{\mathrm{Im}}

\begin{document}
\maketitle

\begin{abstract}
In this paper, we show that the analogue of Thurston's asymmetric metric on the Teichm\"uller space of flat structures on the torus is weak Finsler and we give a geometric description of its  unit circle at each point in the tangent space to Teichm\"uller space. 
We then introduce a family of weak Finsler metrics which interpolate between Thurston's asymmetric metric and the Teichm\"uller metric of the torus (which coincides with the the hyperbolic metric). We describe the unit tangent circles of the metrics in this family.

The final version of this paper will appear in Annales Academi\ae  \   Scientiarum Fennic\ae \  Mathematica.
   \medskip
    \end{abstract}

\section{Preliminaries}
\label{Preliminaries}

We shall use the following identification between the Teichm\"{u}ller space of the torus and the upper half-plane model of the hyperbolic plane $\hyperbolic$:

Let $\torus$ be a two-dimensional torus and fix a pair of generators $a, b$ of $\pi_1(\torus)$ represented by two simple closed curves on this surface intersecting at one point.
The Teichm\"{u}ller space of $\torus$, denoted by $\teich(\torus)$, is the set of equivalence classes of pairs $(\Sigma, f)$, where $\Sigma$ is a Riemann surface and $f: \torus \to \Sigma$ a homeomorphism, and where two pairs $(\Sigma_1, f_1), (\Sigma_2, f_2)$ are defined to be equivalent when $f_1 \circ f_2^{-1}$ is isotopic to a biholomorphism.
From the uniformisation theorem, for every point $x$ in $\teich(\torus)$, there is a unique complex number $\zeta$ with $\Im(\zeta) >0$ such that $x$ is represented by the pair $(\complexes/(\integers + \zeta \integers), f)$, where $f$ is a homeomorphism taking the homotopy class of $a,b$ to $1, \zeta \in \integers + \zeta \integers=\pi_1(\complexes/(\integers + \zeta \integers))$ respectively.
In this way, $\teich(\torus)$ is identified with $\hyperbolic=\{z \in \complexes \mid \Im(z) >0\}$. This identification induces an isometry when the Teichm\"uller space  $\teich(\torus)$ is equipped with the Teichm\"uller metric and $\hyperbolic$  is equipped with the metric of constant curvature $-4$. In the sequel,  we shall refer to this metric on  $\hyperbolic$ as the \emph{hyperbolic metric} $d_{hyp}$. The isometry between the  space $\teich(\torus)$ equipped with the so-called Teichm\"uller metric and the space $\hyperbolic$  equipped with the hyperbolic metric is a result of Teichm\"uller, see \cite[Section 9]{Teich} and \cite[Section 9]{TeichTr} for an English translation of Teichm\"uller's paper.

We also need the following notion:

A \emph{weak metric} $\delta$ on a set $X$ is a map $\delta: X\times X\to \reals$ satisfying the following:
\begin{enumerate}
\item  $\delta(x,x)= 0$ for every $x$ in $X$;
\item  $\delta(x,y)\geq 0$ for every $x$ and $y$  in $X$;
\item $\delta(x,y)+\delta(y,z)\geq \delta(x,z)$ for every $x$, $y$ and $z$ in $X$.
\end{enumerate}
In the paper \cite{BPT}, the following weak metric was introduced on $\hyperbolic$:
First, for $\zeta_1, \zeta_2 \in \hyperbolic$, we let 
\begin{eqnarray}
\label{M}
M(\zeta_1, \zeta_2)=\sup_{x\in \reals}\left|\dfrac{\zeta_2-x}{\zeta_1-x}\right|.
\end{eqnarray}
The weak metric $\delta$ is then defined by setting $\delta(\zeta_1, \zeta_2)=\log M(\zeta_1, \zeta_2)$.

In the same paper, the following explicit expression of $\delta$ was obtained:
\begin{equation}
\label{delta}
\delta(\zeta_1, \zeta_2)=\log \left( \frac{|\zeta_2- \bar \zeta_1|+|\zeta_2-\zeta_1|}{|\zeta_1- \bar \zeta_1|}\right).
\end{equation}

We note that this implies that \begin{equation}
\label{asymmetry}
\delta(\zeta_2, \zeta_1)=\delta(\zeta_1, \zeta_2)+\log\frac{\Im (\zeta_1)}{\Im (\zeta_2)}.
\end{equation}

It was also shown that this weak metric has the following two properties:
\begin{enumerate}
\item The arithmetic symmetrisation of the weak metric $\delta$, that is, the weak metric $S\delta$ defined by 
\[ S\delta (\zeta_1,\zeta_2)=\frac{1}{2}\left( \delta (\zeta_1,\zeta_2) + \delta (\zeta_2,\zeta_1)\right)
\]
is a genuine metric and coincides with the hyperbolic metric of the upper half-plane.
\item The weak metric $\delta$ is an analogue for the torus of Thurston's asymmetric metric on Teichm\"uller space.
\end{enumerate}

The last statement needs some explanation, and we give it now.

For any two points $z_1, z_2$ in the Teichm\"{u}ller space $\teich(\torus)$, we take representatives $(\Sigma_1=\complexes/(\integers+\zeta_1 \integers), f_1), (\Sigma_2=\complexes/(\integers+ \zeta_2 \integers), f_2)$, and we regard them as tori equipped with the quotient flat metrics induced by the flat metric of the Euclidean plane.
We set $\delta(z_1, z_2)=\delta(\zeta_1, \zeta_2)$.
In \cite{BPT}, a weak metric  on $\teich(\torus)$ was defined as follows.
Let $\SSS(\torus)$ denote the set of homotopy classes of essential simple closed curves on the torus. We set,
\begin{equation}
\label{kappa}
\kappa(z_1,z_2)=\log \sup_{s \in \SSS(\torus)}\frac{\l_{\Sigma_2}(f_2(s))}{\l_{\Sigma_1}(f_1(s))},
\end{equation}
where $\l$ denotes the length of the closed geodesic in the corresponding homotopy class.
The formula for $\kappa(z_1,z_2)$ is the analogue, in this Euclidean setting, of the formula for Thurston's metric in the hyperbolic setting given in \cite[p. 8]{Thurston1986}.
Theorem 3 of \cite{BPT} says the following:
\begin{equation}
\label{kappa}
\kappa(z_1,z_2)=\delta(\zeta_1,\zeta_2)
\end{equation}
for any $z_1, z_2 \in \teich(\torus)$ and $z_i=(\Sigma_i=\complexes/(\integers+\zeta_i \integers), f_i)$ for $i=1,2$.

The metric $\delta$ has another characterisation which is given in \cite{BPT}. For two metrics $g_1$, $g_2$ on $T_0=\complexes/\integers\oplus i\integers$ and a homeomorphism $\varphi\colon T_0\to T_0$, we define
$$
\mathcal{L}(\varphi)=\sup_{x\ne y}
\left(
\dfrac{d_{g_2}(\varphi(x),\varphi(y))/\l_{g_2}(s)}{d_{g_1}(x,y)/\l_{g_1}(s)}
\right),
$$
where $s$ is a non-trivial simple closed curve, 
and set
$$
\lambda(g_1,g_2)=\inf_\varphi(\log \mathcal{L}(\varphi)).
$$
The function $\lambda$ is invariant under the action of homeomorphisms on $T_0$ homotopic to the identity. Hence, $\lambda$ defines a weak metric on $\teich(\torus)$. The metric is called the \emph{normalised weak Lipschitz distance}. In \cite{BPT}, it was shown that $\kappa(\zeta_1,\zeta_2)=\lambda(\zeta_1,\zeta_2)$
for any $\zeta_1,\zeta_2\in \hyperbolic$.

In the rest of this paper, we investigate further properties of the weak metric $\kappa=\delta$. We first show that the  geodesics of the hyperbolic metric of $\hyperbolic$ are geodesics with respect to this weak metric. We then show that this metric is weak Finsler (in a sense we shall make precise) and we give a geometric description of its  unit circle at each point in the tangent space to Teichm\"uller space. 
We then introduce a family of weak Finsler metrics which interpolates between the weak metric $\delta$ and the hyperbolic metric (which coincides with the Teichm\"uller metric) which arises naturally from the construction given in this paper.  We describe the unit tangent circle at each point for each weak metric in this family.


\section{Geodesics for the weak metric $\delta$}

In this section, we give an explicit expression for the point where the supremum of  \eqref{M} is attained for given $\zeta_1, \zeta_2 \in \hyperbolic$ and show its geometric meaning.

First we note the following, which can be shown easily from the definition of $\delta$:
\begin{lemma}
\label{invariance}
For $\lambda>0$ and $\tau\in \reals$, we have
\begin{eqnarray}
\delta(\lambda\zeta_1+\tau,\lambda \zeta_2+\tau) &=\delta(\zeta_1, \zeta_2) ,
\label{eq:2}\\
\delta(-\overline{\zeta_1},-\overline{\zeta_2}) &=\delta(\zeta_1, \zeta_2).
\label{eq:3}
\end{eqnarray}
\end{lemma}

For $\zeta_1=a+ib$ and $\zeta_2=\alpha+i\beta$ in $\hyperbolic$ with $\zeta_1\ne \zeta_2$,
we define
\begin{equation}
\label{eq:xpm}
x_\pm 
=\frac{\alpha^2+\beta^2-a^2-b^2}{2(\alpha-a)}\mp
\dfrac{\sqrt{(\alpha-a)^2+(\beta-b)^2}\sqrt{(\alpha-a)^2+(\beta+b)^2}}{2(\alpha-a)}
\end{equation}
if $a\ne \alpha$. 
When $a= \alpha$, we define
\begin{align*}
x_+
&=\begin{cases}
0 & (\beta>b) \\
\infty & (\beta<b),
\end{cases}
\\
x_-
&=\begin{cases}
\infty & (\beta>b) \\
0 & (\beta<b).
\end{cases}
\end{align*}

The following is an explicit expression  for the supremum in \eqref{M}:
\begin{proposition}
\label{maximum}
For $\zeta_1,\zeta_2\in \hyperbolic$ with $\zeta_1\ne \zeta_2$,
the supremum in \eqref{M} is attained at $x_+$.

\end{proposition}

\begin{proof}
Let $\zeta_1=a+ib$ and $\zeta_2=\alpha+i\beta$.
The case where $a=\alpha$ can be easily dealt with. 
The case where $\alpha<a$, from \eqref{eq:3}, by considering $-\overline{\zeta_1}$ and $-\overline{\zeta_2}$ instead of $\zeta_1$ and $\zeta_2$ respectively, is reduced to the case where $a < \alpha$.
Hence we only consider the case where $a<\alpha$.

We first assume that $a=0$.
By assumption, we have $\alpha>0$.
Set
$$
f(x)=\left|\dfrac{\zeta_2-x}{\zeta_1-x}\right|^2=\dfrac{(x-\alpha)^2+\beta^2}{x^2+b^2}
=1+\dfrac{A+Bx}{x^2+b^2},
$$
where $A=\alpha^2+\beta^2-b^2$ and $B=-2\alpha$.
Then,
$$
f'(x)=-\dfrac{Bx^2+2Ax-Bb^2}{(x^2+b^2)^2},
$$
and the critical points of $f'(x)$ are
$$
x_\pm = -\dfrac{A}{B}\pm \dfrac{\sqrt{A^2+b^2B^2}}{B}.
$$
Since ${\rm Re}(\zeta_2)=\alpha>0$ and $B=-2\alpha<0$, we have $x_+<-A/B<x_-$.
Therefore, $f(x)$ attains its maximum at
\begin{align*}
x_+
&=-\dfrac{A}{B} + \dfrac{\sqrt{A^2+b^2B^2}}{B} \\
&=\frac{\alpha^2+\beta^2-b^2}{2\alpha}-\dfrac{\sqrt{\alpha^2+(b-\beta)^2}\sqrt{\alpha^2+(b+\beta)^2}}{2\alpha}.
\end{align*}

Suppose next that $a\ne 0$ and $\alpha>a$.
From the invariance \eqref{eq:2} and the above calculation,  by considering $\zeta_1-a=ib$ and $\zeta_2-a=(\alpha-a)+i\beta$ instead of $\zeta_1$ and $\zeta_2$, we see that the maximum is  attained at 
$$
a+\frac{(\alpha-a)^2+\beta^2-b^2}{2(\alpha-a)}-\dfrac{\sqrt{(\alpha-a)^2+(b-\beta)^2}\sqrt{(\alpha-a)^2+(b+\beta)^2}}{2(\alpha-a)}
$$
which is equal to $x_+$ in \eqref{eq:xpm}.
\end{proof}

\begin{proposition}
\label{endpoint}
The points $x_+$ and $x_-$ in \cref{maximum} are the endpoints at infinity  of the hyperbolic geodesic line  in $\hyperbolic$ passing through $\zeta_2$ and  $\zeta_1$. 
The point $x_+$ lies on the side of $\zeta_1$, and $x_-$ lies on the side of $\zeta_2$. 
\end{proposition}

\begin{proof}
Set $\zeta_1=a+ib$ and $\zeta_2=\alpha+i\beta$ again.
The case where $a=\alpha$ can be easily dealt with. 
Hence, as before, we may assume that $a=0$ and $\alpha>0$, and set $A=\alpha^2+\beta^2-b^2$ and $B=-2\alpha$ as in the proof of \cref{maximum}. Then
\begin{align*}
\left|\zeta_1-\left(-\dfrac{A}{B}\right)\right|^2
&=\left|ib-\left(-\dfrac{A}{B}\right)\right|^2=\frac{A^2}{B^2}+b^2=\frac{A^2+b^2B^2}{B^2}, \\
\left|\zeta_2-\left(-\dfrac{A}{B}\right)\right|^2
&=\left|(\alpha+i\beta)-\left(-\dfrac{A}{B}\right)\right|^2 \\
&=\left(\dfrac{A}{B}+\alpha\right)^2+\beta^2=\left(\dfrac{\alpha^2+\beta^2-b^2}{-2\alpha}+\alpha\right)^2+\beta^2 \\
&=\dfrac{(\alpha^2-\beta^2+b^2)^2+4\alpha^2\beta^2}{4\alpha^2} \\
&=\dfrac{(\alpha^2+(\beta-b)^2)(\alpha^2+(\beta+b)^2)}{4\alpha^2}=\frac{A^2+b^2B^2}{B^2}.
\end{align*}
This means that $-A/B$ is the centre of the Euclidean semicircle  perpendicular to the real axis  passing through the points $\zeta_1$ and $\zeta_2$, and that $x_+, x_0$ are the endpoints of this semicircle. Since $x_+<-A/B<x_-$, $x_+$, $\zeta_1$, $\zeta_2$ and $x_-$ lie on the semicircle in this order.
Since such a semicircle is a hyperbolic geodesic, we have completed the proof.
\end{proof}

\begin{theorem}
\label{geodesic}
Hyperbolic geodesics in $\hyperbolic$ are geodesic with respect to the weak metric $\delta$.
Conversely, every geodesic with respect to $\delta$ is a hyperbolic geodesic.
\end{theorem}
\begin{proof}
Suppose that $\zeta_1, \zeta_2$ and $\zeta_3$ lie on a hyperbolic geodesic $\gamma$ in this order.
By \cref{endpoint}, the endpoint at infinity $x$ of $\gamma$ which lies on the side of $\zeta_1$ not containing $\zeta_2, \zeta_3$ attains the supremum of \eqref{M} for $M(\zeta_1, \zeta_2), M(\zeta_2, \zeta_3)$ and $M(\zeta_1, \zeta_3)$, provided that $x \neq \infty$.
Then by \cref{M}, we have 
$$M(\zeta_1, \zeta_2)=\frac{|\zeta_2-x|}{|\zeta_1-x|},
M(\zeta_2, \zeta_3)=\frac{|\zeta_3-x|}{|\zeta_2-x|},$$ and $$
M(\zeta_1, \zeta_3)=\frac{|\zeta_3-x|}{|\zeta_1-x|}.
$$
This implies that $M(\zeta_1, \zeta_3)=M(\zeta_1, \zeta_2)M(\zeta_2, \zeta_3)$, hence $\delta(\zeta_1, \zeta_3)=\delta(\zeta_1, \zeta_2)+\delta(\zeta_2, \zeta_3)$.
This means that $\gamma$ is a geodesic with respect to $\delta$.

If $x=\infty$, then in the same setting, we have $M(\zeta_1, \zeta_2)=M(\zeta_2, \zeta_3)=M(\zeta_1, \zeta_3)=1$, and again $\gamma$ is a geodesic with respect to $\delta$.

Conversely, suppose that $\gamma$ is a geodesic with respect to $\delta$, and let $\zeta_1, \zeta_2, \zeta_3$ be arbitrary three points lying on $\gamma$ in this order.
Then we have $\delta(\zeta_1, \zeta_3)=\delta(\zeta_1, \zeta_2)+\delta(\zeta_2, \zeta_3)$.
By \cref{asymmetry}, this implies that $\delta(\zeta_3, \zeta_1)=\delta(\zeta_3, \zeta_2)+\delta(\zeta_2, \zeta_1)$, hence $S\delta(\zeta_1,\zeta_3)=S\delta(\zeta_1, \zeta_2)+S\delta(\zeta_2, \zeta_3)$ for the arithmetic symmetrisation $S\delta$.
Since $S\delta$ coincides with the hyperbolic metric, we see that $\gamma$ is also a hyperbolic geodesic.
\end{proof}

Before discussing the connection to Teichm\"uller theory, we shall give a brief comment on this heory. The ideal boundary $\partial \hyperbolic=\partial \teich(\torus)$ is canonically identified with the Thurston compactification of $\teich(\torus)$. Recall that the Thurston compactification of $\teich(\torus)$ consists of the projective classes of measured foliations on the base surface (torus) $T_0=\complexes/(\integers+i\integers)$. A \emph{measured foliation} on $T_0$ is an equivalence class of a pair consisting of a foliation on $T_0$ together with a transverse measure. (Note in the general Thurston theory, the foliations may have singular points, where as in the case of the torus that we are discussing, the foliations are without singularities). Two such pairs are equivalent if either they are isotopic. (In the general case, one has to include Whitehead moves in the equivalence relation, but in the case of the torus, there are no such moves.) For $\alpha\in \hat{\reals}$, we define a measured foliation associated with $\alpha$ to be the pair consisting of the foliation obtained as integral curves of unit vectors satisfying $(1+|\alpha|)^{-1}(dx+\alpha dy)=0$ on $T_0$  and the transverse measure defined by $(1+|\alpha|)^{-1}|dx+\alpha dy|$. $\alpha$ is called the \emph{slope} of the foliation.
Notice that when $\alpha=\infty$, the associated measured foliation consists of the integral curves of the unit lines satisfying $dy=0$ equipped with the transverse measure  $|dy|$.

The point $x_+\in \partial \hyperbolic=\partial \teich(\torus)$ discussed at the beginning of this section corresponds to the slope of the horizontal foliation of the Teichm\"{u}ller map from $\zeta_2$ to $\zeta_1$.
Geometrically, the leaves of the horizontal foliation are stretched under the deformation  along the Teichm\"uller geodesic segment from $\zeta_1$ to $\zeta_2$ (see Figure \ref{fig:deformation_horizontal_foliation}).
\begin{figure}
\includegraphics[width = 7cm]{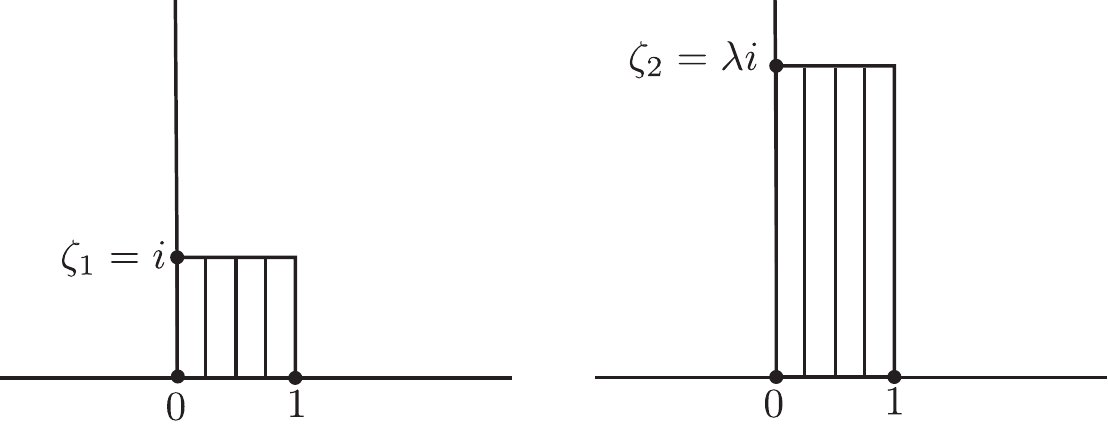}
\caption{The Teichm\"uller ray (the hyperbolic geodesic ray) from $\zeta_1=i$ to $\zeta_2=\lambda i$ with $\lambda>1$ is the vertical ray emanating from from $\zeta_1$. In this case, $x_+=0$, and the leaves of the horizontal foliation are defined by $dx=0$, which is stretched along the deformation from $\zeta_1$ to $\zeta_2$.}
\label{fig:deformation_horizontal_foliation}
\end{figure}
Combining \cref{geodesic} with \cref{kappa}, we have the following.

\begin{corollary}
Suppose that $z, z' \in \teich(\torus)$  correspond to $\zeta_1, \zeta_2 \in \hyperbolic$ respectively.
Then the distance $\kappa(z, z')=\delta(\zeta_1, \zeta_2)$ is attained by the slope of the horizontal foliation for the Teichm\"{u}ller map from $z$ to $z'$.
\end{corollary}

\section{The weak Finsler structure of the weak metric $\delta$}
We recall now the notion of weak norm and weak Finsler metric on a manifold, adapted to the case we are dealing with. 
We start with a \emph{weak norm} on a finite-dimensional vector space $V$. This is a map $V\to [0,\infty)$, $v\mapsto \|v\|$, satisfying
\begin{enumerate}
\item $\|0\|=0$;
 \item $ \|v\|\geq 0$ for all $v$ in $V$;
 \item $\|v+v'\|\leq \|v\|+\|v'\|$ for all $v$ in $V$.
\end{enumerate}

A metric on a smooth manifold $M$ is said to be \emph{weak Finsler} if $M$ is equipped with a continuous field of weak norms defined on the tangent space at each point of $M$ such that the distance between two points in $M$ is equal to the infimum of the lengths of piecewise $C^1$-paths joining them, the length of such a path being computed as the integral over this path of the weak norms of the tangent vectors.
 
In this section, we show that the weak metric $\delta$ on $\teich(\torus)$ is weak Finsler and we give a description of its induced weak norm on the  tangent space of each point in this space.
We start with the following proposition.
\begin{proposition}
\label{norm}
Let $\zeta$ be a point in $\hyperbolic$, and $v$ a tangent vector at $\zeta$. The weak metric $\delta$ induces on $v$ a weak norm  $\|v\|_\delta$ expressed by
$$\|v\|_\delta=\dfrac{|v|+\Im(v)}{2\Im(\zeta)}.$$
\end{proposition}

The meaning of the expression ``induced weak norm" will be clear from the computation done in the proof, and it acquires its complete significance in \cref{Finsler} which follows.
\begin{proof}
Set $\zeta'=\zeta+tv$ ($t>0$). Then,
\begin{align*}
|\zeta'-\zeta|&=|\zeta+tv-\zeta|=t|v|; \\
|\zeta'-\overline{\zeta}|&
= |\zeta-\overline{\zeta}+tv| = |\zeta-\overline{\zeta}|\left|1+t\dfrac{v}{\zeta-\overline{\zeta}}\right| \\
&=|\zeta-\overline{\zeta}|\left(1+t{\rm Re}\left(\dfrac{v}{\zeta-\overline{\zeta}}\right)
+o(t)\right)\\
&=|\zeta-\overline{\zeta}|\left(1+t{\rm Re}\left(\dfrac{v}{2i{\rm Im}(\zeta)}\right)
+o(t)\right)\\
&=2{\rm Im}(\zeta)\left(1+t\dfrac{{\rm Im}(v)}{2{\rm Im}(\zeta)}
+o(t)\right).
\end{align*}
Hence, we have
\begin{align*}
\delta(\zeta,\zeta+tv)
&=\log\dfrac{1}{2{\rm Im}(\zeta)}\left(t|v| +2{\rm Im}(\zeta)\left(1+t\dfrac{{\rm Im}(v)}{2{\rm Im}(\zeta)}
+o(t)\right)\right) \\
&=\log\left(1+t\dfrac{|v|+{\rm Im}(v)}{2{\rm Im}(\zeta)}
+o(t)\right)\\
&=t\dfrac{|v|+{\rm Im}(v)}{2{\rm Im}(\zeta)}
+o(t).
\end{align*}
Thus, we obtain
$$
\lim_{t\searrow +0}\dfrac{\delta(\zeta,\zeta+tv)}{t}=\dfrac{|v|+{\rm Im}(v)}{2{\rm Im}(\zeta)}.
$$
\end{proof}
Notice that as an invariant expression, the weak metric in \cref{norm} is presented as
\begin{equation}
\label{invariant_expression}
\|\cdot\|_\delta=\dfrac{\sqrt{dx^2+dy^2}+dy}{2y}=ds_{hyp}+\dfrac{1}{2}d\log y 
\end{equation}
on $T_\zeta\hyperbolic$ and $\zeta=x+iy\in \hyperbolic$, where $ds_{hyp}$ is the length element of the hyperbolic metric on $\hyperbolic$ of constant curvature $-4$.

Consider now the hyperbolic space $ \hyperbolic$ equipped with the hyperbolic distance (of constant curvature $-4$)  and let $\gamma$ be an isometrically parametrised geodesic emanating from a point $\zeta_0$ in $ \hyperbolic$ and converging to a point $x_0\in \partial \hyperbolic$. We recall that the associated \emph{Busemann function} associated with the geodesic ray $\gamma$ is defined by
$$
\hyperbolic \ni \zeta\mapsto \lim_{t\to \infty}\left(d_{hyp}(\zeta,\gamma(t))-t\right)
$$
(cf. \cite[Chapter 12]{MR3156529}).
Combining this with \cref{geodesic}, we have the following corollary.

\begin{corollary}
\label{Finsler}
The weak metric space $(\teich(\torus), \delta)$ is a weak Finsler metric space with the corresponding weak norm $\|\cdot \|_\delta$ given in \cref{norm}.
\end{corollary}

\begin{proof}
We first show that
\begin{equation}
\label{integration_1}
\int_{\theta_1}^{\theta_2}\|\dot{\gamma}(\theta)\|_\delta d\theta\ge \delta(\zeta_1,\zeta_2)
\end{equation}
for $\zeta_1,\zeta_2\in \hyperbolic$ and any piecewise $C^1$-path $\gamma\colon [\theta_1,\theta_2]\to \hyperbolic$ connecting $\zeta_1$ to $\zeta_2$. Indeed, from the invariant expression \eqref{invariant_expression}, the integration in the left-hand side of \eqref{integration_1} is at least equal to the hyperbolic distance between $\zeta_1$ and $\zeta_2$ minus the difference of the Busemann functions at $\zeta_1$, $\zeta_2$ (see \eqref{Busemannfunction}), which is the right-hand side of \eqref{invariant_expression} (cf. \S\ref{sec:Deforming_Teichmuller_metric}). This observation also implies that the integration in the left-hand side of \eqref{integration_1} is minimised only when it is done along the hyperbolic geodesic from $\zeta_1$ to $\zeta_2$.

We now show that the distance between any two points $\zeta_1$ and $\zeta_2\in \hyperbolic$ is given by integrating the weak norm $\|\cdot \|_\delta$ along a parametrised geodesic joining these two points.

To this end, we first assume that ${\rm Re}(\zeta_1)\ne {\rm Re}(\zeta_2)$.
As in the proof of \cref{maximum}, we may assume that $\zeta_1=ib$ and $\zeta_2=\alpha+i\beta$ with $\alpha>0$. Let $A=\alpha^2+\beta^2-b^2$, $B=-2\alpha$ and $R=\sqrt{A^2+b^2B^2}/|B|$. Define $\theta_1,\theta_2\in (0,2\pi)$ by $e^{i\theta_1}=(\zeta_1-(-A/B))/R$ and $e^{i\theta_2}=(\zeta_2-(-A/B))/R$. 
Note that $\theta_2 < \theta_1$.
The geodesic from $\zeta_1$ to $\zeta_2$ is parametrised as $\gamma(\theta)=(-A/B)+Re^{i(\theta_1+\theta_2-\theta)}$ ($\theta_2\le \theta\le \theta_1$).  Hence by setting $\phi$ to be $\theta_1+\theta_2-\theta$, we have
\begin{align*}
\int_{\theta_2}^{\theta_1}\|\dot{\gamma}(\theta)\|_\delta d\theta
&=\int_{\theta_2}^{\theta_1}\dfrac{|\dot{\gamma}(\theta)|+{\rm Im}(\dot{\gamma}(\theta))}{2{\rm Im}(\gamma(\theta))}d\theta =\int_{\theta_2}^{\theta_1}\dfrac{R-R\cos (\theta_1+\theta_2-\theta)}{2R\sin(\theta_1+\theta_2-\theta)}d\theta \\
&=\int_{\theta_2}^{\theta_1}\dfrac{1-\cos(\phi)}{2\sin(\phi)}d\phi
=\dfrac{1}{2}\log(1+\cos\theta_2)-\dfrac{1}{2}\log(1+\cos \theta_1) \\
&=\dfrac{1}{2}\left(\log\left(1-\dfrac{2\alpha^2-A}{\sqrt{A^2+b^2B^2}}\right)-\log\left(1+\dfrac{A}{\sqrt{A^2+b^2B^2}}\right)\right).
\end{align*}
An easy calculation shows that this expression is equal to 
\[\log\dfrac{\sqrt{\alpha^2+(\beta-b)^2}+\sqrt{\alpha^2+(\beta+b)^2}}{2b}\]
that is, to $\delta(\zeta_1,\zeta_2).$
This gives what we wanted.

We now suppose that ${\rm Re}(\zeta_1)={\rm Re}(\zeta_2)$. Then $\delta(\zeta_1,\zeta_2)$ is equal to zero when ${\rm Im}(\zeta_2)<{\rm Im}(\zeta_1)$ and to $\delta(\zeta_1,\zeta_2)=d_{hyp}(\zeta_1,\zeta_2)$ otherwise.
From \eqref{invariant_expression}, $\|v\|_\delta$ is equal to zero when ${\rm Im}(v)<0$ and to $\|v\|_{hyp}:=\sqrt{ds^2_{hyp}(v,v)}$ otherwise. Hence, the integral of the $\delta$-norm along the hyperbolic geodesic connecting from $\zeta_1$ to $\zeta_1$ coincides with the $\delta$-distance from $\zeta_1$ to $\zeta_2$.
\end{proof}

Next we describe the unit circle in the tangent space with respect to the weak norm $\|\cdot \|_\delta$.

\begin{proposition}
\label{sphere}
The unit circle of the tangent space at $\zeta\in \hyperbolic$ with respect to $\|\cdot \|_\delta$ is expressed as a parabola with focus at the origin and vertex at $i{\rm Im}(\zeta)$.
\end{proposition}
\begin{proof}
Let $\zeta=\alpha+i\beta$.
When $v=v_1+iv_2\in \complexes\cong T_\zeta\hyperbolic$ (as real vector spaces) lies on the unit circle of the tangent space at $\zeta$, we have
$$
1=\dfrac{|v|+{\rm Im}(v)}{2{\rm Im}(\zeta)}=\dfrac{\sqrt{v_1^2+v_2^2}+v_2}{2\beta},
$$
which is equivalent to
$$
v_1^2+v_2^2=(2\beta-v_2)^2=4\beta^2-4\beta v_2 +v_2^2.
$$
This means that the unit tangent circle at $\zeta=\alpha+i\beta\in \hyperbolic$ is the parabola
$$
v_2 = -\dfrac{v_1^2}{4\beta}+\beta,
$$
which implies the desired result.
\end{proof}

Note that the fact that the unit tangent circle of the weak Finsler norm has an infinite direction expresses the fact that the distance function is degenerate in this direction (that is, we have, in this direction,  $\delta(x,y)=0$ for   $x\not=y$).

\section{Deforming $\delta$ to the Teichm\"{u}ller metric}
\label{sec:Deforming_Teichmuller_metric}
In this section, we consider a family of weak Finsler metrics which interpolate between $\delta$ and the hyperbolic distance (which, as is well known,  coincides with the Teichm\"uller distance). We then describe the unit tangent circle of each of these metrics.

Consider the family of weak metrics $\delta_p$ $(0\le p\le 1)$ defined by
\begin{align}
\delta_p(\zeta_1,\zeta_2)
&=\log 
\sup_{x\in \reals}
\dfrac{({\rm Im}(\zeta_1))^{p/2}}{({\rm Im}(\zeta_2))^{p/2}}\left|\dfrac{\zeta_2-x}{\zeta_1-x}\right| \nonumber \\
&=
\dfrac{p}{2}
\log\dfrac{{\rm Im}(\zeta_1)}{{\rm Im}(\zeta_2)}+
\log 
\sup_{x\in \reals}\left|\dfrac{\zeta_2-x}{\zeta_1-x}\right|.
\label{eq:distance_p}
\end{align}
Note that the function
\begin{equation}
\label{Busemannfunction}
\hyperbolic\ni \zeta\mapsto
\dfrac{1}{2}\log\dfrac{{\rm Im}(\zeta_0)}{{\rm Im}(\zeta)}=-\dfrac{1}{2}\int_{\zeta_0}^{\zeta}d\log y
\end{equation}
is the Busemann function associated with the geodesic ray emanating some fixed point $\zeta_0\in \hyperbolic$ converging to $x=\infty\in \partial \hyperbolic$ of the hyperbolic metric of curvature $-4$, which is the Teichm\"uller distance. Hence, the function
$$
\dfrac{1}{2}\log\dfrac{{\rm Im}(\zeta_1)}{{\rm Im}(\zeta_2)}
=
\dfrac{1}{2}\log\dfrac{{\rm Im}(\zeta_0)}{{\rm Im}(\zeta_2)}
-
\dfrac{1}{2}\log\dfrac{{\rm Im}(\zeta_0)}{{\rm Im}(\zeta_1)}
$$
appears in \eqref{eq:distance_p} is the difference of the Busemann functions.
Using the same proof as in \cref{geodesic}, the hyperbolic geodesic from $\zeta_1$ to $\zeta_2$ is the geodesic of the metric $\delta_p$. The arithmetic symmetrisation of $\delta_p$ is the hyperbolic metric of curvature $-4$, like for $\delta=\delta_0$ (cf. (1) in \S\ref{Preliminaries}).

As we did in \cref{norm}, we can calculate the infinitesimal form of the metric $\delta_p$:

Let $\zeta\in \hyperbolic$. For $v\in T_\zeta\hyperbolic\cong \complexes$,
\begin{align*}
\log\dfrac{{\rm Im}(\zeta)}{{\rm Im}(\zeta+tv)}
&= -\log\left(
1+t\dfrac{{\rm Im}(v)}{{\rm Im}(\zeta)}\right) =-t\dfrac{{\rm Im}(v)}{{\rm Im}(\zeta)}+o(t)
\end{align*}
as $t\searrow +0$. We obtain
\begin{align*}
\|v\|_{\delta_p}:=\lim_{t\searrow +0}\dfrac{\delta_p(\zeta,\zeta+tv)}{t}
&=-\dfrac{p}{2}\dfrac{{\rm Im}(v)}{{\rm Im}(\zeta)}+\dfrac{|v|+{\rm Im}(v)}{2{\rm Im}(\zeta)}
\\
&=
\dfrac{|v|+(1-p){\rm Im}(v)}{2{\rm Im}(\zeta)} \\
&=\|v\|_{\delta}-p\dfrac{{\rm Im}(v)}{2{\rm Im}(\zeta)}.
\end{align*}
Notice that $\|v\|_{\delta_p}>0$ when $v\ne 0$ and $p>0$.
The unit tangent circle with respect to the weak norm $\|\cdot \|_{\delta_p}$ in the tangent space $T_\zeta\hyperbolic$ is the ellipse with foci $0$ and $-4{\rm Im}(\zeta)(1-p)/(p(2-p))$ (see Figure \ref{fig:infinitesimal_circle}).
\begin{figure}
\includegraphics[width = 5cm]{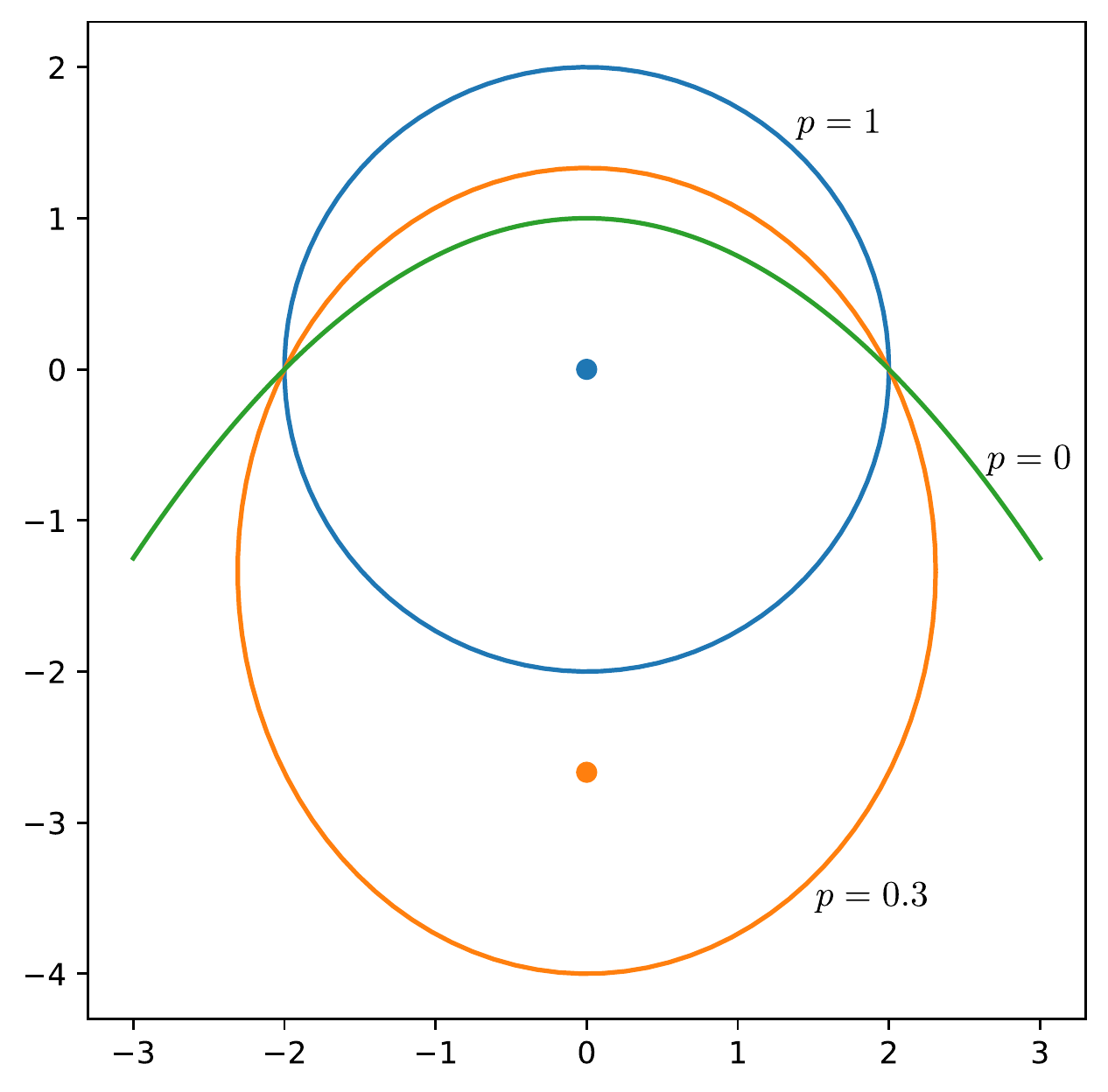}
\caption{The infinitesimal unit circle of the weak norm $\|\cdot \|_{\delta_p}$ at $\zeta=0$ for $p=1$, $0.3$, and $0$.
Each infinitesimal unit circle has the origin as a focus, and the lower dot is another focus for the infinitesimal unit circle for $p=0.3$.}
\label{fig:infinitesimal_circle}
\end{figure}
As an invariant expression, the weak metric $\|\cdot\|_{\delta_p}$ is presented as
\begin{equation}
\label{invariant_expression2}
\|\cdot\|_{\delta_p}=\dfrac{\sqrt{dx^2+dy^2}+(1-p)dy}{2y}=ds_{hyp}+\dfrac{1-p}{2}d\log y. 
\end{equation}
%
%
%
%
A discussion similar to that of the proof of \cref{Finsler} and \eqref{invariant_expression2} shows that the weak metric space $(\teich(\torus), \delta_p)$ is a weak Finsler metric with associated weak norm $\|\cdot \|_{\delta_p}$ and
that 
the hyperbolic geodesic from $\zeta_1$ to $\zeta_2\in \hyperbolic$ is a unique geodesic for $\delta_p$.
Notice that $\|\cdot\|_{\delta_1}$ is the norm from the hyperbolic metric. From \eqref{invariant_expression2},
$\|v\|_{\delta_1}\le \|v\|_{\delta_p}^2\le (2-p)\|v\|_{\delta_1}$ for $0<p\le 1$.
Hence $\delta_p$ and $\delta_{q}$ are bi-Lipschitz-equivalent for $0<p,q\le 1$. In particular, for $0<p\le 1$, $\delta_p$ is complete and separates two points in $\hyperbolic$.

It follows that $\{\delta_p\}_{0\le p\le 1}$ is a continuous family of weak Finsler metrics  giving a deformation  from $\delta=\delta_0$ to the hyperbolic metric $\delta_1$ (which is Teichm\"uller metric). Notice that 
$$
\dfrac{|\zeta-x|^2}{{\rm Im}(\zeta)}
$$
coincides with the extremal length of the measured foliation corresponding to $x\in \reals$ up to a constant factor (depending only on $x$). 
(We recall that the extremal length of a simple closed curve $c$ on a Riemann surface is defined to be the infimum of the reciprocals of the moduli of the annuli whose core curves are homotopic to $c$, and the extremal length function can be extended continuously to the space of measured foliations.)
 Hence, $\delta_1$ coincides with Kerckhoff's formula for the Teichm\"uller distance \cite{Ke} adapted to the case of the torus.

 \bigskip
 
 \noindent {\bf Acknowledgements} The authors would like to thank the referee for his careful reading and his insightful remarks and suggestions which improved the paper.

\bibliographystyle{acm}
\bibliography{mop.bib}

\begin{thebibliography}{1}

\bibitem{BPT}
{\sc Belkhirat, A., Papadopoulos, A., and Troyanov, M.}
\newblock Thurston's weak metric on the {T}eichm\"{u}ller space of the torus.
\newblock {\em Trans. Amer. Math. Soc. 357}, 8 (2005), 3311--3324.

\bibitem{Ke}
{\sc Kerckhoff, S.~P.}
\newblock The asymptotic geometry of {T}eichm\"{u}ller space.
\newblock {\em Topology 19}, 1 (1980), 23--41.

\bibitem{MR3156529}
{\sc Papadopoulos, A.}
\newblock {\em Metric spaces, convexity and non-positive curvature},
  second~ed., vol.~6 of {\em IRMA Lectures in Mathematics and Theoretical
  Physics}.
\newblock European Mathematical Society (EMS), Z\"{u}rich, 2014.

\bibitem{Teich}
{\sc Teichm\"{u}ller, O.}
\newblock Extremale quasikonforme {A}bbildungen und quadratische
  {D}ifferentiale.
\newblock {\em Abh. Preuss. Akad. Wiss. Math.-Nat. Kl. 1939}, 22 (1940), 197.

\bibitem{TeichTr}
{\sc Teichm\"{u}ller, O.}
\newblock Extremal quasiconformal mappings and quadratic differentials.
\newblock In {\em Handbook of {T}eichm\"{u}ller theory. {V}ol. {V}, ed. A.
  Papadopoulos}, vol.~26 of {\em IRMA Lect. Math. Theor. Phys.} Eur. Math.
  Soc., Z\"{u}rich, 2016, pp.~321--483.
\newblock Translated from the German by G. Th\'{e}ret.

\bibitem{Thurston1986}
{\sc Thurston, W.~P.}
\newblock Minimal stretch maps between hyperbolic surfaces.
\newblock {\em arxiv: math/9801039\/}.

\end{thebibliography}

\end{document}